\documentclass{amsart}
\usepackage{cite}
\usepackage{amsmath,amssymb,amsthm}
\usepackage[T1]{fontenc}
\usepackage{hyperref}
\usepackage{comment}
\usepackage{tikz,pgf}
\usepackage{arcs}
\usepackage{caption}
\usetikzlibrary{positioning}
\usetikzlibrary{arrows}
\usepackage{todonotes}
\usepackage{float}

\usetikzlibrary{arrows}
\usetikzlibrary{decorations.pathreplacing}
\usetikzlibrary{calc}
\usepackage{graphicx} 

\DeclareMathOperator{\Tran}{Tran}
\DeclareMathOperator{\Sub}{Sub}
\newcommand{\toO}[1]{\xrightarrow{#1}}

\DeclareMathOperator{\core}{core}
\DeclareMathOperator{\hull}{hull}
\newcommand{\cO}{\mathcal{O}}
\newcommand{\Cat}[1]{\operatorname{Cat}(#1)}

\numberwithin{equation}{section}
\newtheorem{theorem}[equation]{Theorem}
\newtheorem{lemma}[equation]{Lemma}
\newtheorem{proposition}[equation]{Proposition}
\newtheorem{example}[equation]{Example}
\newtheorem{corollary}[equation]{Corollary}
\theoremstyle{definition}

\newtheorem{definition}[equation]{Definition}
\newtheorem{notation}[equation]{Notation}
\newtheorem{remark}[equation]{Remark}

\begin{document}

\title{Counting compatible indexing systems for $C_{p^n}$}

\author[M.~A.~Hill]{Michael A. Hill}
\address{UCLA Department of Mathematics, Los Angeles, CA}
\email{mikehill@math.ucla.edu}

\author[J.~Meng]{Jiayun Meng}
\address{Department of Mathematics, University of Texas, Austin, TX}
\email{jiayun@utexas.edu}

\author[N.~Li]{Nan Li}
\address{Department of Mathematics, University of Minnesota, Minneapolis, MN}
\email{li002843@umn.edu}

\thanks{This material is based upon work supported by the National Science Foundation under Grant No. 2105019}

\begin{abstract}
    We count the number of compatible pairs of indexing systems for the cyclic group \(C_{p^n}\). Building on work of Balchin--Barnes--Roitzheim, we show that this sequence of natural numbers is another family of Fuss--Catalan numbers. We count this two different ways: showing how the conditions of compatibility give natural recursive formulas for the number of admissible sets and using an enumeration of ways to extend indexing systems by conceptually simpler pieces.
\end{abstract}

\maketitle

\section{Introduction}
Recent work in equivariant algebra has studied the beautiful variety of different multiplicative structures that can arise when we mix in the action of a finite group. Blumberg--Hill showed that this is a fundamentally combinatorial structure \cite{BHNinfty}. The various multiplicative norms or additive transfers are entirely governed by certain particularly well-behaved subcategories of finite $G$-sets.

This classification was reformulated in independent work of Balchin--Barnes--Roitzheim and of Rubin, where they further underscored the combinatorial structure by showing the norms and transfers are encoded in ``transfer systems'', certain refinements of the poset of subgroups of $G$ under inclusion \cite{BBR} \cite{Rubin21}. 

\begin{definition}[{\cite[Definition 3.4]{Rubin21}, \cite[Definition 7]{BBR}}]
    A transfer system for a finite group \(G\) is a partial order \(\to\) relation on the set \(\Sub(G)\) of subsets of \(G\) such that
\begin{enumerate}
    \item if \(K\to H\), then \(K\leq H\) (so this is a ``weak subposet''),
    \item if \(K\to H\) and \(g\in G\), then \(gKg^{-1}\to gHg^{-1}\), and
    \item if \(K\to H\) and \(J\leq H\), then \(K\cap J\to J\).
\end{enumerate}
Let \(\Tran(G)\) denote the set of transfer systems for \(G\).
\end{definition}

The set of all transfer systems for \(G\) itself has a partial order: we say \(\cO\leq\cO'\) if the identity map is order-preserving. 

\begin{example}\label{exam:Cpn}
    For \(G=C_{p^n}\), the subgroup lattice is order isomorphic to the linear order
    \[
        \{1\leq 2\dots\leq n+1\}=[n+1].
    \]
    The conjugation condition is always satisfied, and the restriction condition here can be rephrased as saying that if \(C_{p^i}\to C_{p^j}\) and \(i\leq k\leq j\), then \(C_{p^i}\to C_{p^k}\).
\end{example}

Balchin--Barnes--Roitzheim showed that the poset of transfer systems for \(C_{p^n}\) is order isomorphic to the Tamari lattice \cite[Theorem 25]{BBR}. 

In this paper, we will focus on the groups \(C_{p^n}\). Example~\ref{exam:Cpn} stresses that we could equivalently look at ``transfer systems for the poset \([n+1]\)'' in the sense of Franchere--Ormsby--Osorno--Qin--Waugh.

\begin{definition}[{\cite[Definition 4.1]{FOOQW}}]
    A transfer system on 
    \[
        [n]=\{1\leq \dots\leq n\}
    \]
    is a weak sub-poset of \([n]\) with partial order \(\to\) that contains all the elements and which satsifies the ``restriction condition'':
    if \(i\to j\) and \(i\leq k\leq j\), then \(i\to k\).
\end{definition}

It is helpful to view these as a graded set in which we allow \(n\) to vary, as this will help encode certain natural operations.

\begin{definition}\label{defn:Tn}
    For each \(n\in\mathbb N\), let
    \[
        T_{n}=\begin{cases}
            \Tran([n]) & n\geq 1\\
            \{\emptyset\} & n=0.
        \end{cases}
    \]
\end{definition}

Balchin--Barnes--Roizheim counted the number of transfer systems for \(C_{p^n}\):
\begin{theorem}[{\cite[Theorem 20]{BBR}}]
    For \(n\), we have 
    \[
        |T_n|=\frac{1}{2n+1}\binom{2n+1}{n}=\Cat{n}.
    \]
\end{theorem}

The numbers ``\(\Cat{n}\)'' are the Catalan numbers. These are ubiquitous in combinatorics, parameterizing structructures from binary rooted trees to Dyck paths. This sequence fits into a bivariant family of sequences.

\begin{definition}
The \emph{Fuss--Catalan numbers} are defined by 
\[
A_n(p,r)=\frac{r}{np+r} \binom{np+r}{n}
\] 
for non-negative $n$ and positive $p$ and $r$.
\end{definition}
The Catalan numbers arise here:
\[
    \Cat{n}=A_n(2,1).
\]
In this paper, we show the next term of the sequence is also related to equivariant algebra. 

Blumberg--Hill studied what kind of compatibility conditions arise if we allow both the additive transfers and multiplicative norms to each be structured by various transfer systems \cite{BHBiIncomplete}. As a slogan, ``the presence of certain multiplicative norms forces some additive transfers''. These conditions were simplified by Chan who gave a definition internal to transfer systems \cite{Chan}. When the conditions are satisfied, we say that \((\cO_a,\cO_m)\) is compatible. Counting these for \(C_{p^n}\) is the main result of this paper.
\begin{theorem}
    For \(C_{p^{n-1}}\), there are
    \[
        A_n(3,1)
    \]
    compatible pairs of transfer systems.
\end{theorem}
We present two proofs of this theorem in Sections~\ref{sec:Recurr} and \ref{sec:Extensions}. Each of which underscores a different combinatorial feature of the number of compatible pairs. The mathematics in this paper arose from an REU project in Summer 2021. The two junior coauthors each came up with a distinct solution to this counting problem, so we include here both of their solutions. A third solution by Henry Ma will appear separately.

\subsection*{Acknowledgements}

The authors thank Bridget Tenner, Kyle Ormsby, Andrew Blumberg, and Angelica Osorno for helpful conversations during this. 

The authors would like to thank the Hausdorff Research Institute for Mathematics for the hospitality in the context of the Trimester program Spectral Methods in Algebra, Geometry, and Topology, funded by the Deutsche Forschungsgemeinschaft (DFG, German Research Foundation) under Germany’s Excellence Strategy – EXC-2047/1 – 390685813.

\section{Operations on Transfer Systems}
\subsection{Concatination}

A key piece of structure on transfer systems for \([n]\) is the ability to ``concatinate'' a transfer system for \([k]\) and one for \([n-k]\) to build one for \([n]\). This makes computations with the entire graded set easier to understand.

\begin{definition}
    Let \(\cO_L\) be a transfer system for \([k]\) and \(\cO_R\) be a transfer system for \([n-k]\). Then we define a relation  \(\cO=\cO_L\oplus\cO_R\) on \([n]\) by saying \(i\toO{\cO}j\) if and only if 
    \begin{enumerate}
        \item \(j\leq k\) and \(i\toO{\cO_L} j\) or
        \item \(i\geq (k+1)\) and \((i-k)\toO{\cO_R} (j-k)\).
    \end{enumerate}
    We call \(\cO\) the {\emph{direct sum}} or {\emph{concatination}} of \(\cO_L\) and \(\cO_R\).
\end{definition}

\begin{remark}
    As a poset, this is just the disjoint union of the two posets \(\cO_L\) and \(\cO_R\). The additional data is the map to \([n]\) in this case.
\end{remark}

\begin{proposition}
    The concatination of transfer systems is a transfer system.
\end{proposition}
\begin{proof}
    The definition of \(\cO\) is that of the disjoint union of posets, so \(\cO\) is a poset, and it visibly maps to the usual inclusions. 
    
    We need only check the restriction condition, and here, it suffices to check the restriction of \(i+k\to j+k\) along \(m\leq j+k\) for some \(m\leq k\). In this case, both the source and target are \(m\).
\end{proof}

The direct sum is a graded operation here:
\[
    T_n\times T_m\xrightarrow{\oplus} T_{n+m}.
\]
We extend this to \(T_0\) by declaring 
\[
    \emptyset\oplus \cO=\cO\oplus\emptyset=\cO
\] 
for any transfer system \(\cO\) or the empty set. 

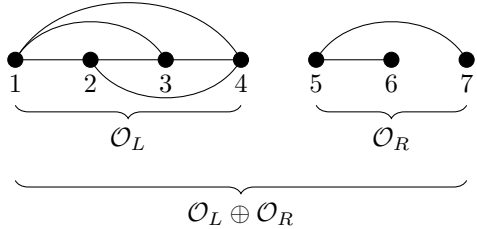
\begin{figure}[ht]
    \centering
    \begin{tikzpicture}
\foreach \x in {1,...,4} {
    \draw[circle,fill] (\x,0)circle[radius=1mm]node[below]{$\x$};
};
\foreach \x in {5,6,7} {
    \draw[circle,fill] (\x,0)circle[radius=1mm]node[below]{$\x$};
};
\foreach \x in {1,2,3,5}{
    \draw(\x,0) to (\x+1,0);
};
\foreach \x in {1,5} {
    \draw(\x,0) to[bend left=60] (\x+2,0);
};
\draw(1,0) to [bend left=60] (4,0);
\draw(2,0) to [bend right=60] (4,0);
\draw [decorate,decoration={brace,amplitude=5pt,mirror,raise=4ex}]
  (1,0) -- (4,0) node[midway,yshift=-3em]{$\cO_L$};
\draw [decorate,decoration={brace,amplitude=5pt,mirror,raise=4ex}]
  (5,0) -- (7,0) node[midway,yshift=-3em]{$\cO_R$};
\draw [decorate,decoration={brace,amplitude=5pt,mirror,raise=4ex}]
  (1,-1) -- (7,-1) node[midway,yshift=-3em]{$\cO_L\oplus\cO_R$};
    \end{tikzpicture}
    \caption{Example of concatination of \(\cO_L\) and \(\cO_{R}\)}
    \label{fig:Concatination}
\end{figure}

\subsection{Restriction}

Given a transfer system for \([n]\), we have two natural ways to build transfer systems for smaller natural numbers. These both arise from the inclusion of subposets.

\begin{definition}
    If \(k\leq n\), then let
    \[
        \iota\colon [k]\to [n]
    \]
    be the map sending \(i\to i\), and let 
    \[
        \phi\colon [n-k]\to [n]
    \]
    be the map sending \(i\) to \(i+k\).
\end{definition}

We can restrict a transfer system on \([n]\) along either of these inclusions.

\begin{definition}
    Let \(\cO\) be a transfer system for \([n]\), and let \(k\leq n\). 
    
    Let \(i_{{k}}^\ast\cO\) be defined by saying for all \(i\leq j\leq k\),
    \[
        i\toO{i_{{k}}^\ast\cO} j \text{ if and only if }
        i\toO{\cO}j.
    \]
    
    Let \(\Phi^{k}\cO\) be defined by saying for all \(i\leq j\leq (n-k)\),
    \[
        i\toO{\Phi^{k}\cO} j\text{ if and only if }
        {(i+k)}\toO{\cO} {(j+k)}.
    \]
\end{definition}

\begin{proposition}
    For a transfer system on \(\cO\),  \(i_{k}^{\ast}\cO\) is a transfer system on \([k]\) and \(\Phi^{k}\cO\) is a transfer system on \([n-k]\).
\end{proposition}
\begin{proof}
    The construction gives wide, subposets of \([k]\) and \([n-k]\) respectively, by observation. Since the two inclusions \(\iota\) and \(\phi\) are interval inclusions, the restriction condition is easily checked.
\end{proof}

The two restrictions can be visualized as simply throwing away any transfers that start or end outside of the given range. An example picture of this is shown in Figure~\ref{fig:Restriction and GFP}.

\begin{remark}
    The map \(i_k^\ast\) is induced by the inclusion of the subgroup lattice corresponding to the inclusion of \(C_{p^k}\) into \(C_{p^n}\). Topologically, this models the restriction to \(C_{p^k}\). The map \(\Phi^{k}\) is induced by the inclusion of subgroup lattices corresponding to the quotient map from \(C_{p^n}\) to \(C_{p^{n-k}}\). Topologically, this models the fixed points.
\end{remark}

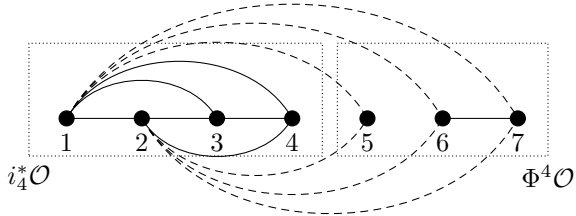
\begin{figure}[ht]
    \centering
    \begin{tikzpicture}
\foreach \x in {1,...,4} {
    \draw[circle,fill] (\x,0)circle[radius=1mm]node[below]{$\x$};
};
\foreach \x in {5,6,7} {
    \draw[circle,fill] (\x,0)circle[radius=1mm]node[below]{$\x$};
};
\foreach \x in {1,2,3,6}{
    \draw(\x,0) to (\x+1,0);
};
\foreach \x in {1} {
    \draw(\x,0) to[bend left=60] (\x+2,0);
};
\draw(1,0) to [bend left=60] (4,0);
\draw(2,0) to [bend right=60] (4,0);
\foreach \x in {5,6,7}{
    \draw[densely dashed](1,0) to [bend left=60] (\x,0);
    \draw[densely dashed](2,0) to [bend right=60] (\x,0);
};
\draw[densely dotted] (0.5,-.5) rectangle (4.4,1);
\draw[densely dotted] (4.6,-.5) rectangle (7.4,1);
\node[below] at (0.5,-.5){$i_4^\ast \cO$};
\node[below] at (7.4,-.5){$\Phi^4\cO$};
    \end{tikzpicture}
    \caption{Example of \(i_k^\ast\cO\) and \(\Phi^k\cO\). The dashed transfers are forgotten to form \(i_4^\ast\cO\) and \(\Phi^4\cO\).}
    \label{fig:Restriction and GFP}
\end{figure}

\subsection{Wrapped and saturated systems}

We single out two families of transfer systems that are useful for our counts.

\begin{definition}
    A transfer system \(\cO\) for \([n]\) is {\emph{wrapped}} if \(1\toO{\cO} n\).
\end{definition}

\begin{example}
    The transfer systems \(\cO_L\) and \(\cO_R\) from Figure~\ref{fig:Concatination} are wrapped; the transfer system \(\cO_L\oplus\cO_R\) is not.
\end{example}

\begin{example}
    The transfer system \(\cO\) in Figure~\ref{fig:Restriction and GFP} is wrapped, as is \(i_{4}^\ast\cO\). The transfer system \(\Phi^4\cO\) is not. 
\end{example}

By the restriction condition, in a wrapped transfer system, we have \(1\to j\) for all \(j\), and we think of the largest transfer as ``wrapping'' the rest of the transfer system. Additionally, we can always ``wrap'' any given transfer system, using the circle-dot product of \cite{BBR}. We use a single case of their construction here.

\begin{definition}
    If \(\cO\in T_n\), let \(w(\cO)\) be the transfer system for \([n+1]\) with
    \begin{enumerate}
        \item for all \(j\), \(1\to j\) in \(w(\cO)\), and
        \item \(\Phi^{1}\big(w(\cO)\big)=\cO\).
    \end{enumerate}
\end{definition}

This is implicit in Balchin--Barnes--Roitzheim's ``circle-dot product'' and their decomposition theorem.
\begin{theorem}[{\cite[Corollary 21]{BBR}}]\label{thm:BBRDecomp}
Any transfer system for \([n]\) with \(n\geq 1\) decomposes uniquely as
\[
    \cO\odot\cO'=\cO\oplus w(\cO'),
\]
for some transfer systems \(\cO\in T_m\) and \(\cO'\in T_{n-m-1}\).
\end{theorem}

\begin{corollary}\label{cor:FullDecomp}
    Any transfer system \(\cO\) for \([n]\) can be written uniquely as
    \[
        \cO=w(\cO_1)\oplus\dots\oplus w(\cO_k)
    \]
    for some transfer systems \(\cO_1,\dots,\cO_k\).
\end{corollary}

A special case of wrapped transfer systems is given by complete ones.

\begin{definition}
    The {\emph{complete}} transfer system for \([n]\) is the one for which \(i\to j\) for all \(i\leq j\). Denote it \(\cO^{cpt}_{n}\).
\end{definition}

\begin{remark}
    The complete transfer system for \([n]\) is one for which the partial order is just the inclusion. This means that complete transfer systems are maximal elements in the poset of transfer systems.
\end{remark}

These complete transfer systems will be especially useful building blocks for us.

\begin{definition}
    A transfer system \(\cO\) is {\emph{saturated}} if it can be written as a direct sum of complete transfer systems.
\end{definition}

%
%

Given any transfer system on \([n]\), there is a minimal saturated transfer system that contains it. This is immediate from the observation that the intersection of two saturated transfer systems is again a saturated transfer system. For concreteness, we spell this out directly here.

\begin{proposition}\label{prop:Hull}
    If \(\cO\) is any transfer system, then there is a minimal saturated transfer system that contains \(\cO\).
\end{proposition}
\begin{proof}
    Write \(\cO\) as a direct sum of wrapped transfer systems \(\cO=\cO_{n_1}\oplus\dots\cO_{n_r}\), with \(\cO_{n_i}\in T_{n_i}\). Then minimal saturated transfer system containing it is simply 
    \[
        \cO_{n_1}^{cpt}\oplus\dots\cO^{cpt}_{n_r},
    \]
    where we replace each summand with the complete one of that size.
\end{proof}

\begin{definition}
    If \(\cO\) is a transfer system for \([n]\), let \(\hull(\cO)\) denote the minimal saturated transfer system that contains it. This is the {\emph{saturated hull}}.
\end{definition}

For transfer systems for \([n]\), there is a kind of dual notion of a maximal saturated transfer systems inside any given transfer system. 

Complete transfer systems have two useful properties:
\begin{enumerate}
    \item Whenever \(i\to j\), we also have \(k\to j\) for all \(i\leq k\leq j\), and
    \item they are generated as a poset as the transitive closure of the relation \(i\to {i+1}\) for all \(1\leq i\leq n-1\).
\end{enumerate}
These two properties give us two different ways to repackage the condition of being saturated. 

\begin{proposition}
    A transfer system \(\cO\) for \(n\) is saturated if and only if whenever \(i\to j\) with \(i<j\), we have \({{(j-1)}}\to {j}\).
\end{proposition}
\begin{proof}
    If \(\cO\) is saturated, then by definition of the direct sum, if \(i\to j\), then \(i\) and \(j\) correspond to subgroups from the same direct summand. By completeness, we therefore have all intermediate transfers.
    
    Using the decomposition of a transfer system into wrapped ones (Corollary~\ref{cor:FullDecomp}), we see that it suffices to show that if a wrapped transfer system that has the property that \(i\to j\) implies \((j-1)\to j\), then the transfer system is complete. This follows from downward induction on \(n\), using that \(1\to n\) by the wrapped assumption, and hence \(1\to j\) for all \(j\) by restriction.
\end{proof}

\begin{corollary}
    A transfer system \(\cO\) is saturated if and only if it is generated as a partial order by relations \(i\to {{(i+1)}}\) for some collection of positive \(i\) at most \(n-1\).
\end{corollary}

The possibly surprising part here is that we only need the partial order: the other parts of being a transfer system come along for free in this case, since we are generating by a covering condition. This gives us a second kind of structural result.

\begin{definition}
    If \(\cO\) is a transfer system let the {\emph{core}} of \(\cO\), denoted \(\core{\cO}\), be the partial order generated by \({i}\to {{(i+1)}}\), where \(i\) ranges over the integers from \(1\) to \(n-1\) such that \(i\to(i+1)\).
\end{definition}

\begin{example}
    In Figure~\ref{fig: Core Example}, we see the core of the transfer system \(\cO\) from Figure~\ref{fig:Restriction and GFP}.

\begin{figure}[ht]
    \centering
    \begin{tikzpicture}
\foreach \x in {1,...,4} {
    \draw[circle,fill] (\x,0)circle[radius=1mm]node[below]{$\x$};
};
\foreach \x in {5,6,7} {
    \draw[circle,fill] (\x,0)circle[radius=1mm]node[below]{$\x$};
};
\foreach \x in {1,2,3,6}{
    \draw(\x,0) to (\x+1,0);
};
\foreach \x in {1} {
    \draw(\x,0) to[bend left=60] (\x+2,0);
};
\draw(1,0) to [bend left=60] (4,0);
\draw(2,0) to [bend right=60] (4,0);
    \end{tikzpicture}
    \caption{Core of \(\cO\) from Figure~\ref{fig:Restriction and GFP}}
    \label{fig: Core Example}
\end{figure}

\end{example}

\section{Compatible pairs}
Our main object of study is the notion of compatible pairs. These were defined by Blumberg--Hill to describe ``compatibility'' between equivariant norms and transfers in an abstract, categorical way \cite{BHBiIncomplete}. Chan reformulated this in the language of transfer system, giving a purely combinatorial formulation \cite[Theorem 4.10]{Chan}. We use that here for the special case of \(C_{p^n}\).

\begin{definition}[{\cite[Definition 4.6]{Chan}}]\label{def:Compatible}
    A pair of transfer system \((\cO_a,\cO_m)\) for \([n]\) are compatible if for whenever \(i\xrightarrow{m} j\), we have \(k\xrightarrow{a} j\) for all \(i\leq k\leq j\).
\end{definition}

Note here that the conditions are asymmetrical: arrows in \(\cO_m\) force those in \(\cO_a\). Moreover, this is a kind of relative saturation condition, with an arrow in \(\cO_m\) actually forcing \(\cO_a\) to look saturated in a range. This gives us two equivalent forms.

\begin{proposition}\label{prop:CompatibleCoreHull}
    A pair \((\cO_a,\cO_m)\) is compatible if and only if the following equivalent comparisons hold
    \begin{enumerate}
        \item \(\cO_m\leq \core(\cO_a)\),
        \item \(\hull(\cO_m)\leq\cO_a\), and
        \item \(\hull(\cO_m)\leq\core(\cO_a)\).
    \end{enumerate}
\end{proposition}
\begin{proof}
    The conditions of Definition~\ref{def:Compatible} are a restatement of the condition that the saturated hull of \(\cO_m\) is less than or equal to \(\cO_a\). For the equivalence of the three conditions, we use that the core of \(\cO\) is the largest saturated transfer system less than or equal to \(\cO\) and the hull is the smallest saturated transfer system greater than or equal to \(\cO\). 
\end{proof}

\begin{corollary}
    Let \((\cO_a,\cO_m)\) be a compatible pair of transfer systems for \([n]\). If for some \(1\leq k\leq n-1\), we have \(k\not\toO{a} (k+1)\), then for all \(j\leq k\) and \(\ell\geq (k+1)\), we must have \(j\not\toO{m} \ell\).
\end{corollary}

Put another way, we see that \(\cO_m\) must break apart at \(k\), and this must be compatible with \(\cO_a\).

\begin{corollary}\label{cor:CompatiblesBreak}
    Let \((\cO_a,\cO_m)\) be a compatible pair of transfer systems for \([n]\). If for some \(1\leq k\leq (n-1)\), we have \(k\not\toO{a} (k+1)\), then 
    \begin{enumerate}
        \item \(\cO_m=i_{k}^{\ast}\cO_m\oplus\Phi^{k}\cO_m\), and
        \item the pairs 
        \[
        (i_{k}^\ast\cO_a,i_{k}^{\ast}\cO_m)\text{ and }(\Phi^{k}\cO_a,\Phi^{k}\cO_m)
        \]
        are compatible.
    \end{enumerate}
\end{corollary}

\begin{definition}
Let
\[
    \mathcal D^{(0)}_n=\big\{(\cO_a,\cO_m)\mid\text{compatible}\big\}\subset T_n\times T_n.
\]
\end{definition}

We have projection maps
\[
    p_a,p_m\colon\mathcal D^{(0)}_n\to T_n
\]
which take a pair \((\cO_a,\cO_m)\) to \(\cO_a\) or \(\cO_m\) respectively. Our main goal is to find the cardinality of \(\mathcal D^{(0)}_n\) for all \(n\). We solve this in several different ways using different aspects of Proposition~\ref{prop:CompatibleCoreHull}.

\section{Solving the recurrence relations}\label{sec:Recurr}

\subsection{Decomposition and Recurrence Relation} \label{subsec decom recur}

The wrapping map \(w\) defines a natural filtration on the collection of transfer systems. We can use this to build a recursive relation describing compatible pairs.

\begin{definition}
    For each \(i\geq 0\), let
    \[
        (\mathcal F^i T)_n= Im(w^{\circ i})\subset T_n,
    \]
    viewed as a graded subset.
\end{definition}

\begin{definition}
    Let 
    \[
        \mathcal D^{(i)}_n=p_a^{-1}\big((\mathcal F^i T)_n\big)
    \]
    be the set of composable pairs with \(\cO_a\in \mathcal F^i T\). 
    
    Let \(d(n,i)=|\mathcal D^{(i)}_n|\) be the corresponding cardinality. 
\end{definition}

We deduce our recursive formulae from the Balchin--Barnes--Roitzheim decomposition theorem (Theorem~\ref{thm:BBRDecomp}). We restate the result here to set up our decomposition.
\begin{proposition}\label{prop: Refined Decomposition}
    If \(\cO\) is a transfer system in \((\mathcal F^i T)_n\), then there is 
    \begin{enumerate}
        \item a unique natural number \(1\leq j\leq n-i\),
        \item a unique wrapped transfer system \(w\cO_R\) in \(T_j\), 
        \item and a unique transfer system \(\cO_L\) in \(T_{n-i-j}\)
    \end{enumerate}   
    such that
    \[
        \cO=w^i\big(\cO_L\oplus w\cO_R\big).
    \]
\end{proposition}

\begin{notation}
    Let \((\mathcal F^i T)_{n,j}\) be the set of transfer systems in \((\mathcal F^i T)_n\) which decompose as
    \[
        \cO=w^i\big(\cO_L\oplus w\cO_R\big)
    \]
    with \(w\cO_R\in T_j\) a wrapped transfer system.
\end{notation}

\begin{proposition}
    The map
    \[
        \left(\coprod_{j=1}^{n-i-1} \mathcal D^{(i)}_{n-j}\times \mathcal D^{(1)}_{j}\right)\amalg \mathcal D^{(i+1)}_n
        \to \mathcal D^{(i)}_n
    \]
    given on \(\mathcal D^{(i)}_{n-j}\times \mathcal D^{(1)}_{j}\) by
    \[
        \big((w^i\cO_a',\cO_m'),(w\cO_a'',\cO_m'')\big)\mapsto \big(w^i(\cO_a'\oplus w\cO_a''),\cO_m'\oplus\cO_m''\big)
    \]
    and on the last summand by the natural inclusion, is a bijection.
\end{proposition}
\begin{proof}
We use Proposition~\ref{prop: Refined Decomposition} to further break up \(\mathcal D^{(i)}_n\), since the decomposition here gives a disjoint union decomposition
\[
    (\mathcal F^i T)_n = \coprod_{j=1}^{n-i} (\mathcal F^i T)_{n,j}.
\]
This decomposition induces a decomposition of \(\mathcal D^{(i)}_n\):
\[
    \mathcal D^{(i)}_{n,j}=p_1^{-1} (\mathcal F^i T)_{n,j}.
\]
Since \(T_0=\{\emptyset\}\), the unit for \(\oplus\), we have
\[
    \mathcal D^{(i)}_{n,n-i}=\mathcal D^{(i+1)}_{n},
\]
given by the usual inclusion. Now let \(1\leq j<n-i\), and consider an element \((\cO_a,\cO_m)\) in \(\mathcal D^{(i)}_{n,j}\). By definition, we have
\[
    \cO_a=w^i\big(\cO_{a,L}\oplus w\cO_{a,R}\big),
\]
with \(\cO_{a,L}\neq\emptyset\) and \(w\cO_{a,R}\) wrapped, and hence  we are missing the transfer
\[
    (i+n-j)\to (i+n-j+1)
\]
in \(\cO_a\). This means that \(\cO_m\) breaks up into a direct sum 
\[
    \cO_m'\oplus\cO_m'',
\]
where \(\cO_m'\in T_{i+n-j}\) and \(\cO_m''\in T_{j}\), by Corollary~\ref{cor:CompatiblesBreak}. Moreover, we know that the pairs
\[
\big(i_{i+n-j}^{\ast}\cO_a,\cO_m'\big)\text{ and }\big(\Phi^{i+n-j}\cO_a,\cO_m''\big)
\]
are compatible. The result follows, since
\[
    i_{i+n-j}^{\ast}\cO_a=w^i\cO_{a,L}\text{ and }\Phi^{i+n-j}\cO_a=w\cO_{a,R}.
\]
\end{proof}

\begin{corollary} \label{cor recursive formula}
    We have a recursive formula
    \[
        d(n,i)=d(n,i+1)+\sum_{j=1}^{n-i-1} d(n-j,i) d(j,1).
    \]
\end{corollary}

The base case here is actually \(\mathcal D^{(n)}_n\), which is \(p_a^{-1} (\cO_n^{cpt})\). Every transfer system is compatible with the additive one. Note also that we have an important edge case: 
\[
(\mathcal F^{n-1}T)_n=(\mathcal F^n T)_n,
\]
since both correspond to the unique complete transfer system on \([n]\).

\begin{proposition}[{\cite[Theorem 20]{BBR}}]\label{prop: Count Base Case}
    For each \(n\), we have
    \[
        d(n,n)=d(n,n-1)=\Cat{n}.
    \]
\end{proposition}

\subsection{Rewriting the Recurrence Relation}

We now can solve the recurrence relation, giving our first proof of the main theorem. Recall the definition of the Fuss--Catalan numbers:
\[
	A_n(p,r)=\frac{r}{np+r}\binom{np+r}{n}.
\]
We begin with some helpful properties of the Fuss-Catalan number.

\begin{proposition}\label{Prop of Fuss}
For any positive integer $p$ and positive $n,r$, the following properties hold.

\begin{enumerate}
    \item\label{Splitting An(3,r)} $A_n(3,r)= \sum_{j=0}^{n} A_j(2,1)A_{n-j}(3,j+r-1)$.
    \item\label{Dropping n} $A_n(p,r)=A_n(p,r-1)+A_{n-1}(p,p+r-1)$.
    \item\label{Identifying An(3,3)} $A_{n+1}(p,1)=A_n(p,p)$.
    \item\label{Splitting at r} $A_n(p,s+r)=\sum_{j=0}^{n} A_j(p,r)A_{n-j}(p,s)$.
    \item\label{Splitting An(3,2)} $A_n(3,2)= \sum_{j=1}^{n} A_{j+1}(2,1)A_{n-j}(3,j)$.
\end{enumerate}
\end{proposition}

\begin{proof}
Four of these formulae are from work of M{\l}otkowski. Formula \ref{Splitting An(3,r)} is a special case of \cite[Proposition 2.1]{Mlotkowski2010}. The proof \ref{Dropping n} is \cite[Equation 2.2]{Mlotkowski2010}, \ref{Identifying An(3,3)} is \cite[Equation 2.3]{Mlotkowski2010}, and \ref{Splitting at r} is \cite[Equation 2.4]{Mlotkowski2010}. 

We will only prove \ref{Splitting An(3,2)}. It can be done by induction on \(n\).  The base case $n=1$ is a straightforward check. Take $n\geq 2$. 

Denote $$s=\sum_{j=1}^n A_{j+1}(2,1)A_{n-j}(3,j).$$ Replacing \(j\) with \(j+1\), we can rewrite this as 
\[
s=
\sum_{j=2}^{n+1} A_{j}(2,1)A_{n-j+1}(3,j-1).
\]
Applying \ref{Dropping n} twice, we have 
\begin{align*}
A_{n-j+1}(3,j+1)&=
A_{n-j+1}(3,j)+A_{n-j}(3,j+3) \\
&=\big(A_{n-j+1}(3,j-1)+A_{n-j}(3,j+2)\big)+A_{n-j}(3,j+3).
\end{align*}

With the help of \ref{Splitting An(3,r)}, we can expand $A_{n+1}(3,2)$ to
\[
A_{n+1}(3,2)= A_{n+1}(3,1)+A_{n}(3,2)+
\sum_{j=2}^{n+1} A_{n-j+1}(3,j+1)A_j(2,1),
\]
and substituting in for \(A_{n-j+1}(3,j+1)\), we can rewrite this as
\begin{multline*}
A_{n+1}(3,2) = A_{n+1}(3,1)+A_{n}(3,2)+ s+\\ \sum_{j=2}^{n} A_{n-j}(3,j+2)A_j(2,1)
+\sum_{j=2}^{n} A_{n-j}(3,j+3)A_j(2,1).
\end{multline*} 
The last two sums also show up in the expansions of \(A_n(3,3)\) and \(A_n(3,4)\), respectively, by \ref{Splitting An(3,r)}:
\begin{align*} 
\sum_{j=2}^{n} A_{n-j}(3,j+2)A_j(2,1)
=A_n(3,3)-A_n(3,2)-A_{n-1}(3,3),\\
\sum_{j=2}^{n} A_{n-j}(3,j+3)A_j(2,1)
=A_n(3,4)-A_n(3,3)-A_{n-1}(3,4).
\end{align*}
This gives an equality
\begin{multline}\label{eqn: ReIdentifying s}
A_{n+1}(3,2)=A_{n+1}(3,1)+A_{n}(3,2)+s\\ +\big(A_{n}(3,3)-A_{n}(3,2)-A_{n-1}(3,3)\big)+\big(A_{n}(3,4)-A_n(3,3)-A_{n-1}(3,4)\big).
\end{multline}
We can use \ref{Identifying An(3,3)} to rewrite Equation~\ref{eqn: ReIdentifying s} as
\[
A_{n+1}(3,2)=A_{n+1}(3,1)+A_{n}(3,4)+\big(s-A_{n}(3,1)-A_{n-1}(3,4)\big).
\]
However, applying \ref{Dropping n} to \(A_{n+1}(3,2)\), we get exactly
\[
A_{n+1}(3,2)=A_{n+1}(3,1)+A_{n}(3,4),
\]
so we deduce
\[
s=A_{n}(3,1)+A_{n-1}(3,4).
\]
By \ref{Dropping n} again, we get $s=A_n(3,2)$.
\end{proof}

\begin{theorem}\label{thm number of orange}
For any $n\in \mathbb{N}$, $A_{n-1}(3,2)$ is $d(n,1)$.
\end{theorem}
Before proving the theorem, we first discuss the strategy of the proof. 
Instead of focusing on the recurrence steps for a specific $n$, we put all $d(k,j)$ together which forms a triangle as follows.

\begin{figure}[H]
    \centering
    \begin{tikzpicture}

        \node at (0,0) {$d(2,1)$};
        \node at (0,-0.5) {$d(3,1)$};
        \node at (0,-1) {$d(4,1)$};
        \node at (0,-1.5) {$d(5,1)$};
        \node at (0,-2) {...};
        
        \node at (1.5,-0.5) {$d(3,2)$};
        \node at (1.5,-1) {$d(4,2)$};
        \node at (1.5,-1.5) {$d(5,2)$};
        
        \node at (3,-1) {$d(4,3)$};
        \node at (3,-1.5) {$d(5,3)$};
        
        \node at (4.5,-1.5) {$d(5,4)$};

    \end{tikzpicture}

\end{figure}

Recall Corollary \ref{cor recursive formula}. For $i\in \mathbb{N}_{+}$, the elements in the set $\{d(j+i, j): j\geq 1\}$ share the same recursive formula.
In fact, we can generate the above triangle hypotenuse by hypotenuse from outside to inside. To begin with, we generate the second outermost hypotenuse $d(3,1)$, $d(4,2)$, $d(5,3)$... by the outermost hypotenuse $d(2,1)$, $d(3,2)$, $d(4,3)$... according to the recursive formula for $d(n,n-2)$. In other word, $d(n,n-2)$ is a linear combination of $d(j,j-1)$ with coefficients 1 or $d(1,1)$.
Similarly, we can generate the third outermost hypotenuse by the outermost and the second outermost hypotenuse according to the recursive formula for $d(n,n-3)$. Now the coefficients are from the space spanned by $d(2,1)$, $d(1,1)$ and 1. 

As the whole triangle can be generated hypotenuse by hypotenuse, we can conclude that any $d(k,l)$ in the triangle can be written as a linear combination of $d(j, j-1)$ where $l+1\leq j\leq k$ with coefficients from the space spanned by $1$ and $d(i,1)$ where $1\leq i\leq k-l-1$. Graphically, $d(k,l)$ is a linear combination $d(j,j-1)$ in the blue area.

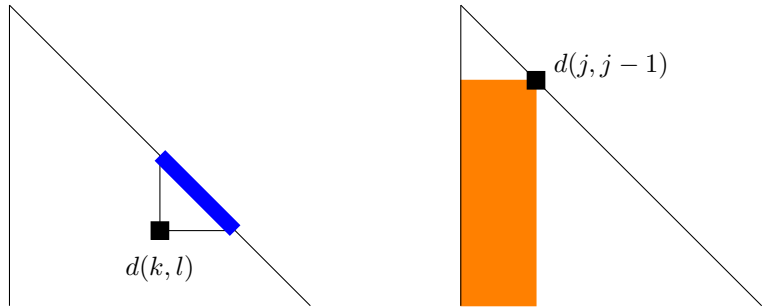
\begin{figure}[H]
    \centering
    \begin{tikzpicture}

        \draw (0,0) -- (0,4);
        \draw (0,4) -- (4,0);
        \node [fill, draw] at (2,1) {};
        \draw (2,1) -- (2,2);
        \draw (2,1) -- (3,1);
        \node at (2, 0.5) {$d(k,l)$};
        \draw[line width=0.2cm, blue] (2,2) -- (3,1);
        
        \draw[fill=orange, orange] (7,3) rectangle (6,0);

        \draw (6,0) -- (6,4);
        \draw (6,4) -- (10,0);
        \node [fill, draw] at (7,3) {};
        \node at (8, 3.2) {$d(j,j-1)$};

    \end{tikzpicture}
    \label{fig: Orienting the double induction}
    \caption{Showing what contributes to each term: \(d(k,l)\) is a linear combination of \(d(j,j-1)\) from the blue shaded region. And a given \(d(j,j-1)\) contributes only in the orange region.}
\end{figure}

\begin{proof}[Proof of Theorem~\ref{thm number of orange}]
It is easy to check for $d(2,1)$. Then by induction, it suffices to show the theorem is true for $d(n,1)$ if the theorem holds for $d(i,1)$ for any $2\leq i\leq n-1$.

Fix $j$ such that $2\leq j\leq n$. For our convenience, we denote the coefficient of $d(j,j-1)$ for $d(k,l)$ as $c(k,l)$. Our goal is to determine $c(n, 1)$. We will achieve this by induction. 

First, we figure out the coefficients of $d(j,j-1)$ for the vertical line $d(j+l, j-1)$ where $0\leq l\leq n-j$. We know $c(j, j-1)=1=A_0(3,1)$ and for $0\leq l\leq n-j$,
$$c(j+l,j-1)=\sum_{i=1}^{l} d(i,1)\, c(j+l-i,j-1).$$
As $d(i,1)$ is $A_{i-1}(3,2)$ for $2\leq i\leq n-1$, we get $c(j+l, j-1)=A_{l}(3,1)$ by \ref{Splitting at r} of Proposition \ref{Prop of Fuss} and the equality \(A_n(3,3)=A_{n+1}(3,1)\) for $0\leq l\leq n-j$ inductively. 

In fact, we claim that for $1\leq i\leq j-1$, $c(j+l, j-i) = A_l(3,i)$ for $0\leq l\leq n-j$. We have shown the case for $i=1$. Now it is sufficient to show the case $k$ provided the claim holds for cases $1\leq i\leq k-1$. Similarly, we know $c(j, j-k)=1=A_0(3,k)$ and for $0\leq l\leq n-j$,
$$c(j+l,j-k)=c(j+l, j-(k-1))+\sum_{i=1}^{l} d(i,1)\, c(j+l-i,j-k).$$
As $d(i,1)$ is $A_{i-1}(3,2)$ for $2\leq i\leq n-1$, we get $c(j+l, j-k)=A_{l}(3,k)$ by \ref{Dropping n} and \ref{Splitting at r} of Proposition \ref{Prop of Fuss} for $0\leq l\leq n-j$ inductively. Thus, we get $c(n,1)=A_{n-j}(3,j-1)$. As $d(n,1)$ depends only on $d(j,j-1)$ for $2\leq j\leq n$,
$$
d(n,1)=\sum_{j=2}^{n} A_{n-j}(3,j-1)d(j,j-1)=\sum_{j=2}^{n}A_{n-j}(3,j-1) A_{j}(2,1),
$$ where the last equality is 
Proposition~\ref{prop: Count Base Case}. By \ref{Splitting An(3,2)} of Proposition \ref{Prop of Fuss}, we conclude $d(n,1)=A_{n-1}(3,2)$.
\end{proof}

\begin{theorem}
There are $A_{n}(3,1)$ compatible transfer systems for $[n]$.
\end{theorem}
\begin{proof}
We will prove this by induction. It is easy to check the case for $[1]$. Assume the theorem is true for $[i]$ where $1\leq i<n$. We want to show it is true for $[n]$.

Corollary~\ref{cor recursive formula} shows that the number of compatible systems for $[n]$ is 
\[
d(n,0)=d(n,1)+\sum_{j=1}^{n-1}d(j,0) d(n-j,1).
\]
Theorem~\ref{thm number of orange} and the inductive hypothesis let us rewrite this as
\[
d(n,0)=A_{n-1}(3,2)+\sum_{j=1}^{n-1} A_{j}(3,1)A_{n-j-1}(3,2)=\sum_{j=0}^{n-1} A_{j}(3,1)A_{n-j-1}(3,2).
\]
By \ref{Splitting at r} of Proposition~\ref{Prop of Fuss}, we deduce
\[
d(n,0)=A_{n-1}(3,3),
\] 
which by \ref{Identifying An(3,3)} of Proposition~\ref{Prop of Fuss}, equals $A_{n}(3,1)$. 
\end{proof}

\section{Extensions of saturated systems}\label{sec:Extensions}
\subsection{Counting using additive cores}
Instead of using the filtration by powers of \(w\) on the additive indexing system, we can use the first part of Proposition~\ref{prop:CompatibleCoreHull}. This says that compatibility of \((\cO_a,\cO_m)\) is the same question as compatibility of \((\core\cO_a,\cO_m)\), since both reduce to the comparison
\[
    \cO_m\leq\core\cO_a.
\]
Since the core breaks up as a direct some of complete transfer systems, this last condition is really the same as asking that if 
\[
    \core\cO_a=\cO_{n_1}^{cpt}\oplus\dots\oplus\cO_{n_k}^{cpt},
\]
then we have a direct sum decomposition
\[
    \cO_m=\cO_{n_1}'\oplus\dots\oplus\cO_{n_k}',
\]
where \(\cO_{n_i}'\in T_{n_i}\) for all \(i\). This is the only condition here, so we deduce the following proposition.

\begin{proposition}\label{prop:CountByCores}
    If \(\cO_a\) is a transfer system with
    \[
        \core\cO_a=\cO_{n_1}^{cpt}\oplus\dots\oplus\cO_{n_k}^{cpt},
    \]
    then there are
    \[
        \prod_{i=1}^k \Cat{n_i}
    \]
    transfer systems \(\cO_m\) such that the pair \((\cO_a,\cO_m)\) is compatible.
\end{proposition}

This reduces our question of the number of compatible pairs to two parts:
\begin{enumerate}
    \item Enumerate all of the transfer systems with a fixed core, then
    \item evaluate the corresponding sum.
\end{enumerate}

\begin{notation}
    Let \(\vec{k}=(k_1,\dots,k_n)\) be a sequence of positive integers. 
    For each \(1\leq s\leq n\), let
    \[
        K_s=k_1+\dots+k_s.
    \]
    Let
    \[
        \cO_{\vec{k}}^{sat}=\cO_{k_1}^{cpt}\oplus\dots\oplus\cO_{k_n}^{cpt}.
    \]
\end{notation}

\begin{definition}
	For a sequence \(\vec{k}=(k_1,\dots,k_n)\) with \(k=k_1+\dots+k_n\), let
	\[
		\mathcal E_{\vec{k}} = \big\{\cO\in T_k\mid\core{\cO}=\cO_{\vec{k}}^{sat}\big\},
	\]
	and let
	\[
		e_{\vec{k}} = |\mathcal E_{\vec{k}}|.
	\]
\end{definition}

\subsection{Catalan tuples}
The enumeration and exact sum were analyzed by de Jong, Hock, and Wulkenhaar  \cite{MR4407998} in a slightly different guise. They consider certain sequences which they call Catalan tuples.
\begin{definition}[{\cite[Definition 3.1]{MR4407998}}]\label{def: Catalan Tuple}
    For each positive integer \(n\), a {\emph{Catalan tuple of length \(n\)}} is a a sequence of non-negative integers
    \[
        \vec{s}=(s_0,\dots,s_r)
    \]
    with three properties:
    \begin{enumerate}
        \item for all \(j\), we have
        \[
            \sum_{i=0}^{j}s_i >j,
        \]
        \item at the end,
        \[
            \sum_{i=0}^{r}s_i=n,
        \]
        \item and if \(n>0\), then \(s_r>0\).
\end{enumerate}
    Let \(\mathcal S_n\) be the set of all Catalan tuple of length \(n\).
\end{definition}
\begin{remark}
    We have slightly modified the definition here to ignore trailing zeros. This removes our ability to predict the length of the string, but it will better connect with the extensions.
\end{remark}

We can restate the conditions in Definition~\ref{def: Catalan Tuple} slightly to start connecting with extensions. 

\begin{definition}
    The \emph{excess} of a Catalan tuple \(\vec{s}=(s_1,\dots,s_r)\) of length \(n\) is 
    \[
    e(\vec{s}):=n-r-1.
    \]
\end{definition}

\begin{remark}
    Note that the edge condition in Definition~\ref{def: Catalan Tuple} of \(j=r\) implies the inequality \(n>r\). This means the excess is always non-negative.
\end{remark}

\begin{proposition}\label{prop: Extending Catalan Tuples}
    Let \(\vec{s}\) be a Catalan tuple of length \(n\) and excess \(e=n-r-1\). Then for any \(k>0\), the sequence
    \[
        \vec{s}_{k}(\ell)=(s_0,\dots,s_r,\underbrace{0,\dots,0}_{\ell},k)
    \]
    is a Catalan tuple if and only if 
    \[
        0\leq \ell\leq e.
    \]
\end{proposition}
\begin{proof}
    For \(0\leq j\leq r\), the Catalan tuple condition holds since it does for \(\vec{s}\). If \(\ell=0\), then we have 
    \[
        \sum_{i=0}^{r} s_i+k=n+k>r+1,
    \]
    since \(n>r\) and \(k\geq 1\). Assume now that \(\ell>0\), and consider a \(1\leq j\leq \ell\). We have
    \[
        \sum_{i=0}^{r+j} s_i=n=r+e+1.
    \]
    On the other hand, this is greater than \(r+j\) if and only if \(j<e+1\). This gives the bounds on \(\ell\). Finally, note that the analysis for the case \(\ell=0\) now also implies the case \(j=\ell+1\).
\end{proof}

\begin{proposition}\label{prop: Excess of Catalan Extensions}
    Let \(\vec{s}\) be a Catalan tuple of length \(n\) and excess \(e=n-r-1\). Then for any \(k>0\) and \(0\leq \ell\leq e\), the excess of the Catalan tuple
    \[
        \vec{s}_{k}(\ell)=(s_0,\dots,s_r,\underbrace{0,\dots,0}_{\ell},k)
    \]
    is 
    \[
        e\big(\vec{s}_{k}(\ell)\big)=k-1+e-\ell=n+k-(r+\ell+2).
    \]
\end{proposition}
\begin{proof}
    The Catalan tuple given has length \(n+k\). The sequence \(\vec{s}\) has length \((r+1)\), and we added \(\ell+1\) new terms to form \(\vec{s}_k(\ell)\).
\end{proof}

This lets us rewrite Catalan tuples using only the non-zero entries.

\begin{definition}
    If \(\vec{s}\) is a Catalan tuple, then let \(\core(\vec{s})\) be the subsequence of non-zero entries of \(\vec{s}\). Given a partition \(\vec{k}\) of \(k\), let \(\mathcal S_{\vec{k}}\) be the subset of \(\mathcal S_k\) of Catalan tuples with core \(\vec{k}\):
    \[
        \mathcal S_{\vec{k}}:=\{\vec{s}\in\mathcal S_k\mid \core{\vec{s}}=\vec{k}\}.
    \]
\end{definition}

\begin{corollary}
    Catalan tuples with core \(\vec{k}\) are those sequences of the form
    \[
        k_1,\underbrace{0,\dots,0}_{\ell_1},k_2,\underbrace{0,\dots,0}_{\ell_2},k_3,\dots,k_{n-1},\underbrace{0,\dots,0}_{\ell_{n-1}} k_n
    \]
    such that for all \(1\leq j\leq n-1\), we have
    \[
        \sum_{i=1}^{j} (k_i-\ell_i)\geq j.
    \]

    The excess of such a sequence is
    \[
        (k_n-1)+\sum_{i=1}^{n-1}(k_i-1-\ell_i).
    \]
\end{corollary}

In their work, de Jong, Hock, and Wulkenhaar consider certain collections of Catalan tuples.
\begin{definition}[{\cite[Definition 4.1]{MR4407998}}]
    A {\emph{nested Catalan tuple of length \(n\)}} is a sequence of Catalan tuples \((\vec{s}_{i_1},\dots,\vec{s}_{i_r})\) such that \(\vec{s}_{i_j}\in\mathcal S_{i_j}\) and the sequence 
    \[
        (i_1+1,i_2,\dots,i_r)
    \]
    is a Catalan tuple of length \(n\).
\end{definition}

A key result in \cite{MR4407998} is the cardinality of the number of nested Catalan tuple that begin with \((0)\). For this, we need a straightforward lemma.
\begin{lemma}
    The map \(\Sigma\colon \mathcal S_n\to \mathcal S_{n+1}\) given by
    \[
        (s_0,\dots,s_r)\mapsto (1,s_0,\dots,s_r)
    \]
    is an injection with image those sequences which begin with \(1\).
\end{lemma}

\begin{proposition}[{\cite[Corollary 4.6]{MR4407998}}]
    The number of nested Catalan tuples of length \((n+1)\) with first term \((0)\) is
    \[
        \sum_{\vec{s}\in\mathcal S_n} \prod_i\Cat{s_i}=\frac{1}{2n+1}\binom{3n}{n}=\frac{1}{3n+1}\binom{3n+1}{n}=A_n(3,1).
    \]
\end{proposition}

We will produce an explicit bijection
\[
    \sigma\colon T_n\to\mathcal S_n
\]
by building bijections between \(\mathcal E_{\vec{k}}\) and \(\mathcal S_{\vec{k}}\).

\subsection{Enumerating Extensions by a complete transfer system}

It is helpful to think of elements of \(\mathcal E_{\vec{k}}\) also as various ``extensions'' of the complete transfer systems \(\cO_{k_1}^{cpt}\),\dots \(\cO_{k_n}^{cpt}\). 
For our count, it is easier to instead consider a more general class.
\begin{definition}
    An {\emph{extension of \(\cO'\in T_m\) by \(\cO''\in T_{k}\)}} is a transfer system \(\cO\in T_{m+k}\) such that
    \begin{enumerate}
        \item \(i_m^\ast \cO=\cO'\),
        \item \(\Phi^m \cO=\cO''\), and
    \end{enumerate}
    
    an extension \(\cO\) of \(\cO'\) by \(\cO''\) is  {\emph{core-preserving}} if moreover:
    \[
        \core{\cO}=\core{\cO'}\oplus \core{\cO''}.
    \]
\end{definition}

Note that since by assumption we have specified \(i_m^\ast\) and \(\Phi^m\) in an extension, we need only determine the transfers with source \(i\leq m\) and target \(j>m\).

\begin{definition}
    Any transfer \(i\to j\) with \(i\leq m\) and \(j>m\) in an extension of \(\cO'\in T_m\) by \(\cO''\) is a {\emph{crossing-transfer}}.
\end{definition}

\begin{proposition}
    Let \(\cO\) be an extension of \(\cO'\in T_m\) by \(\cO''\in T_{k}\). Then following are equivalent:
    \begin{enumerate}
        \item The extension is core-preserving.
        \item If there is a transfer \(m\to j\), then \(j=m\).
    \end{enumerate}
\end{proposition}
\begin{proof}
  Note that the existence of a nontrivial transfer \(m\to j\) is equivalent to the existence of a transfer \(m\to (m+1)\), by the restriction axiom. If we have
  \[
    \core{\cO'}=\cO_{n_1}^{cpt}\oplus\dots\oplus\cO_{n_j}^{cpt}\text{ and }\core{\cO''}=\cO_{m_1}^{cpt}\oplus\dots\cO_{m_i}^{cpt},
  \]
  then by construction of the core, the existence of the transfer \(m\to (m+1)\) is equivalent to the core of \(\cO\) be
  \[
    \core{\cO}=\cO_{n_1}^{cpt}\oplus\dots\oplus \cO_{n_{j-1}}^{cpt}\oplus\cO_{n_j+m_1}^{cpt}\oplus \cO_{m_2}\oplus\dots\oplus \cO_{m_i}.
  \]
  The result follows.
\end{proof}

Since we want to enumerate transfer systems with a fixed core, we now restriction attention to core-preserving extensions of \(\cO\) by a complete transfer system \(\cO'\). This significantly simplifies our combinatorics.

\begin{lemma}\label{lem: Crossing Transfers}
    Let \(\cO\) be a core-preserving extension of \(\cO'\in T_m\) by \(\cO_{k}^{cpt}\). Then for each \(i\leq m\), the following are equivalent
    \begin{enumerate}
        \item We have a crossing transfer \(i\to m+j\) for some \(j>0\).
        \item We have crossing transfers \(i\to m+j\) for all \(0\leq j\leq k\).
    \end{enumerate}
\end{lemma}
\begin{proof}
    One direction is immediate. For the other, if we have a transfer \(i\to m+j\), then by the restriction axiom, we have transfers \(i\to m\) and \(i\to m+1\). Since \(\cO_{k}^{cpt}\) is complete, in our extension, we have transfers \(m+1\to m+j\) for any \(1\leq j\leq k\), which gives the second result.
\end{proof}

\begin{remark}
    We singled out the transfer \(i\to m\) since this constrains the number of possible sources for a transfer from \([m]\) up to \([k]\). Any crossing transfer has source an element of \([m]\) that transfers up to \(m\) in \(\cO'\).
\end{remark}

Moreover, however, we have a kind of ``non-decreasing'' property.

\begin{lemma}\label{lem: Crossing Nondecreasing}
    Let \(\cO\) be a core-preserving extension of \(\cO'\in T_m\) by \(\cO_{k}^{cpt}\), and let 
    \[
        \{d_1<d_2<\dots<d_r<d_{r+1}=m\mid d_i\to m\}
    \]
    be the set of elements of \([m]\) which transfer to \(m\) in \(\cO'\). In \(\cO\), if we have a transfer 
    \[
        d_j\to m+k,
    \]
    then for all \(1\leq i\leq j\), we have transfers
    \[
        d_i\to m+k.
    \]
\end{lemma}
\begin{proof}
    By the restriction axiom, whenever \(i\leq j\), we have a transfer \(d_i\to d_j\). The result follows from transitivity.
\end{proof}

\begin{definition}
    Let \(\cO'\in T_m\), and let 
    \[
        \{d_1<d_2<\dots<d_r<d_{r+1}=m\mid d_i\to m\}
    \]
    be the set of elements of \([m]\) which transfer to \(m\) in \(\cO'\). For each \(k>0\) and for each \(0\leq \ell\leq r\), define a relation \(\to_{\ell}\) on \([m+k]\) that refines the partial order \(\leq\) by saying
    \begin{enumerate}
        \item if \(i\leq j\leq m\), then \(i\to_{\ell} j\) if and only if \(i\to j\) in \(\cO'\),
        \item if \(m<i\leq j\), then \(i\to_{\ell} j\) if and only if \(i\leq j\), and 
        \item if \(i\leq m <j\), then \(i\to_{\ell} j\) if and only if \(i=d_s\) for some \(1\leq s\leq r-\ell\).
    \end{enumerate}
\end{definition}

\begin{proposition}\label{prop: Building Extensions}
    The relation \(\to_{\ell}\) is a transfer system on \([n+k]\) that is a core-preserving extension of \(\cO'\) by \(\cO_{k}^{cpt}\).
\end{proposition}

\begin{definition}
    Let \(\cO_{\cO',k}(\ell)\) denote the transfer system \(\to_\ell\) on \([n+k]\).
\end{definition}

\begin{remark}
    There is a special case of the extensions: \(\ell=r\). In this case, we have the direct sum \(\cO'\oplus \cO_{k}^{cpt}\).
\end{remark}

There is a crucial observation about the number of transfers to \([m+k]\) here.

\begin{proposition}\label{prop: Transfers in an Extension}
    Let \(\cO'\) be a transfer system on \([m]\) in which \(r+1\) elements transfer up to \(m\), and let \(0\leq \ell\leq r\). Then in \(\cO_{\cO',k}(\ell)\), we have \(k+r-\ell\) elements which transfer up to \(m+k\).
\end{proposition}
\begin{proof}
    All of the \(k\) elements of \([k]\) transfer up, and by construction, the \(r-\ell\) elements \(d_1,\dots, d_{r-\ell}\) are the only elements from \([m]\) which also transfer up to \(m+k\).
\end{proof}

Putting these together gives a complete classification of the core-preserving extensions.

\begin{theorem}
    Let \(\cO\in T_m\) be a transfer system, and let 
    \[
        \{d_1<d_2<\dots<d_r<d_{r+1}=m\mid d_i\to m\}
    \]
    be the set of elements of \([m]\) which transfer to \(m\) in \(\cO'\). Then there are \((r+1)\) core-preserving extensions of \(\cO\) by \(\cO_{k}^{cpt}\) given by \(\cO_{\cO',k}(\ell)\) for \(0\leq \ell\leq r\).
\end{theorem}
\begin{proof}
    Lemma~\ref{lem: Crossing Transfers} and Lemma~\ref{lem: Crossing Nondecreasing} show that any core-preserving extension has this form. The converse is the content of Proposition~\ref{prop: Building Extensions}.
\end{proof}

\subsection{Enumerating transfer systems with a fixed core}

Now let \(\cO\) be a transfer system with 
\[
    \core{\cO}=\cO_{\vec{k}}^{sat}. 
\]
Write \(\vec{k}=(k_1,\dots,k_{n})\). We can immediately identify \(\cO\) inductively as a type considered in the previous section. 
\begin{proposition}
    The transfer system \(\cO\) is a core-preserving extension of \(i_{K_{n-1}}^\ast\cO\) by \(\cO_{k_n}^{cpt}\).
\end{proposition}
This turns our problem into an inductive one, working down on the number of summands in the partition of \(k\). We can now build our bijection.

\begin{definition}
    Let 
    \[
        \sigma\colon \mathcal E_{\vec{k}}\to\mathcal S_{\vec{k}}
    \]
    be defined inductively by the following procedure. 
    If \(\vec{k}=(k)\), then \(\cO=\cO_k^{cpt}\), and we define
    \[
        \sigma(\cO_{k}^{cpt})=k.
    \]
    
    For a general \(\vec{k}=(k_1,\dots,k_n)\) with \(n>1\) and \(\cO\in\mathcal E_{\vec{k}}\), let \(\ell_{n-1}\) be the unique number \(0\leq \ell_{n-1}\) such that
    \[
        \cO=\cO_{\cO',k_n}(\ell_{n-1}),
    \]
    where \(\cO'=i_K^\ast \cO\), and define
    \[
        \sigma(\cO)=\sigma(\cO'),\underbrace{0,\dots,0}_{\ell_{n-1}}, k_n.
    \]
\end{definition}

\begin{example}
    For \(\cO\) the transfer system in Figure~\ref{fig:Restriction and GFP}, we have
    \[
        \sigma\cO=(4,0,1,2).
    \]
\end{example}

We need to verify that \(\sigma\) actually lands in the set \(\mathcal S_{\vec{k}}\).

\begin{proposition}
    For any \(\cO\) with \(\core{\cO}=\vec{k}\), we have
    \begin{enumerate}
        \item \(\sigma(\cO)\in\mathcal S_{\vec{k}}\), and
        \item \(e\big(\sigma(\cO)\big)\) is the number of elements \(j\) between \(1\) and \(K_n-1\) that transfer up to \(K_n\) in \(\cO\).
    \end{enumerate}
\end{proposition}
\begin{proof}
    We show this by induction on \(n\). The base case of \(n=1\) is immediate by the definitions of \(\sigma\) and the excess, so assume this is true for partitions with fewer than \(n\) terms. 
    
    Let \(\cO'=i_{K_{n-1}}^\ast\cO\), and let \(r\) be the number of \(j<K_{n-1}\) which transfer up to \(K_{n-1}\) in \(\cO'\). By the inductive hypothesis, \(\sigma\cO'\) is a Catalan tuple with core \((k_1,\dots,k_{n-1})\) and we also have
    \[
        e\big(\sigma\cO'\big)=r.
    \]
    Now if \(0\leq \ell_n\leq r\) is such that \(\cO=\cO_{\cO',k_n}(\ell_{n-1})\), then since \(\ell_{n-1}\leq e\big(\sigma\cO'\big)\), 
    the first claim is Proposition~\ref{prop: Extending Catalan Tuples}. For the second part, Proposition~\ref{prop: Excess of Catalan Extensions} shows that the excess of \(\sigma\cO\) is 
    \[
        e(\sigma\cO)=e-\ell_{n-1}+(k_{n}-1).
    \]
    Proposition~\ref{prop: Transfers in an Extension} shows this quantity is exactly the number of elements smaller than \(K_n\) which transfer up to \(K_n\).
\end{proof}

This gives us the final piece for our argument.

\begin{corollary}
    The map \(\sigma\) is a bijection \(\mathcal E_{\vec{k}}\to\mathcal S_{\vec{k}}\).
\end{corollary}
\begin{proof}
    By induction on \(n\), we see that there are exactly as many extensions of \(i_{K_{n-1}}^\ast\cO\) by \(\cO_{k_n}\) as there are extensions of the Catalan tuple \(\sigma(i_{K_{n-1}}^\ast\cO)\) to a Catalan tuple ending with \(k_n\), and the map \(\sigma\) gives a bijection between these.
\end{proof}

\bibliographystyle{abbrv}
\bibliography{refs}

\end{document}